\documentclass[reqno,11pt]{amsart}

\usepackage[all]{xy}
\CompileMatrices
\usepackage{amssymb}
\usepackage{mathrsfs}
\usepackage{cite}
\usepackage{amsfonts}
\usepackage[active]{srcltx}

\renewcommand{\to}{\longrightarrow}

\newcommand{\rad}{\mathop{\mathrm{rad}}}

\newcommand{\Ind}{\mathop{\mathrm{Ind}}\nolimits}
\newcommand{\Coind}{\mathop{\mathrm{Coind}}\nolimits}

\newcommand{\inv}{^{-1}}
\newcommand{\p}{\varphi}

\newcommand{\ov}[1]{\ensuremath{\overline {#1}}}
\newcommand{\til}[1]{\ensuremath{\widetilde {#1}}}

\newcommand{\Irr}{\mathop{\mathrm{Irr}}\nolimits}

\newcommand{\soc}[1]{\mathrm{Soc}(#1)}

\newcommand{\End}{\mathop{\mathrm{End}}\nolimits}
\newcommand{\Hom}{\mathop{\mathrm{Hom}}\nolimits}
\newcommand{\Ext}{\mathop{\mathrm{Ext}}\nolimits}

\newtheorem{Thm}{Theorem}[section]
\newtheorem{Prop}[Thm]{Proposition}

\newtheorem{Lemma}[Thm]{Lemma}
{\theoremstyle{definition}
}
{\theoremstyle{remark}
}
\newtheorem{Cor}[Thm]{Corollary}
{\theoremstyle{remark}
}

{\theoremstyle{remark}
}
{\theoremstyle{remark}
}
{\theoremstyle{remark}
}

\numberwithin{equation}{section}

\title[Projective indecomposable modules and quivers]{Projective indecomposable modules and quivers for monoid algebras}

\author{Stuart Margolis\ and Benjamin Steinberg}
\address{Department of Mathematics\\
Bar Ilan University\\ 52900 Ramat Gan \\ Israel \and Department of Mathematics\\
    City College of New York\\
    Convent Avenue at 138th Street\\
    New York, New York 10031\\
    USA}
\thanks{The authors were supported by United States-Israel Binational Science Foundation \#2012080 and the second author was supported by NSA MSP \#H98230-16-1-0047.}
\email{margolis@math.biu.ac.il\and bsteinberg@ccny.cuny.edu}
\date{\today}


\begin{document}
\begin{abstract}
We give a construction of the projective indecomposable modules and a description of the quiver for a large class of monoid algebras including the algebra of any finite monoid whose principal right ideals have at most one idempotent generator.  Our results include essentially all families of finite monoids for which this has been done previously, for example, left regular bands, $\mathscr J$-trivial and $\mathscr R$-trivial monoids and left regular bands of groups.
\end{abstract}
\maketitle

\section{Introduction}

In a highly influential paper~\cite{BHR}, Bidigare, Hanlon and Rockmore showed
that a number of popular Markov chains, including the Tsetlin library and the
riffle shuffle, are random walks on the faces of a hyperplane arrangement (the
braid arrangement for these two examples).
More importantly, they showed that the representation theory of the monoid of
faces, where the monoid structure on the faces of a central hyperplane
arrangement is given by the Tits projections~\cite{Titsappendix}, could be used
to analyze these Markov chains and, in particular, to compute the spectrum of
their transition operators.

The face monoid of a hyperplane arrangement satisfies the semigroup identities $x^2 = x$ and $xyx = xy$. Semigroups satisfying these identities are known in the literature as left regular bands (although they were studied early on by Sch\"utzenberger~\cite{Schutzlrb} under the more descriptive name ``\textit{treillis gauches},'' translated by G.~Birkhoff in his Math Review as skew lattices, a term which nowadays has a different meaning). Brown developed~\cite{Brown1,Brown2} a theory of random walks on finite left regular bands. He gave numerous examples that do not come from hyperplane arrangements, as well as examples of hyperplane walks that could more easily be modeled on simpler left regular bands. For example, Brown considered random walks on bases of matroids. Brown used the representation theory of left regular bands to extend the spectral results of Bidigare, Hanlon and Rockmore~\cite{BHR} and gave an algebraic proof of the diagonalizability of random walks on left regular bands.

One of the equivalent definitions of a left regular band is that $M$ is that every element of $M$ is an idempotent and distinct elements of $M$ generate distinct principal right ideals. There are many classes of finite monoids that generalize this property and that have arisen in deep ways in algebra and combinatorics. This paper studies the representation theory of such monoids with an eye  towards applications of these results in related fields.

In this paper we  study the representation theory of a collection of finite monoids that have been called `right semi-abundant' in the literature \cite{Ltilde , ElQallili} but which we dub `right Fountain monoids'  in honor of John Fountain as he has advocated  the study of these and related classes of  monoids for a number of years. Fountain monoids generalize the class of (von Neumann) regular monoids, whose representation theory has been extensively studied~\cite{Putcharep3,rrbg}.  In fact many of the results of this paper are extensions of our previous results~\cite{rrbg} to the non-regular setting.  This requires a bit of work because regular monoids have quasi-hereditary algebras in good characteristic~\cite{Putcharep3}, where as a Fountain monoid can have loops in its quiver.

The most important class of Fountain monoids studied in this paper is the collection of finite monoids such that every principal right ideal has at most one idempotent generator. In the semigroup literature this class is denoted by $\mathbf{ER}$. Besides left regular bands, many of the most widely studied classes of monoids belong to the class $\mathbf{ER}$. We discuss a few of these.
Of course all finite groups belong to $\mathbf{ER}$. More generally, inverse semigroups, which are exactly the class of semigroups such that  every principal right and left ideal has a unique idempotent generator belong to $\mathbf{ER}$. Inverse semigroups  are the regular semigroups that have faithful representations by partial one-to-one maps on a set (for this reason, they appear in many applications of semigroup theory) and the class of regular semigroups whose idempotents commute. It is known that they are precisely the semigroups in $\mathbf{ER}$ whose algebras are semisimple over the complex numbers~\cite{repbook}.  The representation theory of inverse semigroups was studied extensively in~\cite{mobius2}.

The class of $\mathscr{J}$-trivial monoids consist of finite monoids such that each two-sided principal ideal has a unique generator and thus is in $\mathbf{ER}$.
A very important example of a $\mathscr{J}$-trivial monoid defined for any finite Coxeter group is the monoid associated to
its $0$-Hecke algebra. Norton first described the representation theory of the $0$-Hecke algebra
of a Coxeter group $W$ in 1979~\cite{norton}, but did not exploit its structure as the monoid algebra of a
monoid $M(W)$. The monoid $M(W)$ has been rediscovered
many times over the years. The easiest way to define it is as the monoid with generating set the
Coxeter generators of $W$ and relations those of $W$ in braid form, and replacing the involution
relation $s^{2} = 1$ by the idempotent relation $s^{2} = s$ for each Coxeter generator $s$. The monoid
$M(W)$ has a number of amazing properties. Its size is exactly that of $W$ and it admits the
strong Bruhat order as a partial order compatible with multiplication. In fact, it is isomorphic
to the monoid of principal order ideals of the Bruhat order under set multiplication.
From the point of view of this proposal, we are interested that the monoid algebra of $M(W)$
is the $0$-Hecke algebra $H_{0}(W)$, as can readily be seen from the presentation for $M(W)$ that
we mentioned above.  A detailed study of the representation theory of $\mathscr J$-trivial monoids was undertaken in~\cite{Jtrivialpaper}.
More generally, a number of authors have considered the representation of $\mathscr R$-trivial monoids, monoids in which each principal right ideal has a unique generator, especially in connection with Markov chains~\cite{DO,BergeronSaliola,ayyer_schilling_steinberg_thiery.2013}.  Left regular bands are precisely the regular $\mathscr R$-trivial monoids.

Recall that a finite monoid is aperiodic if all its group subsemigroups are trivial. Left regular bands, $\mathscr{J}$-trivial monoids and $\mathscr R$-trivial monoids are aperiodic. There are non-aperiodic analogues of these classes of monoids that have also appeared in the literature. The class $\mathbf{LRBG}$ consists of all finite monoids with the property that every principal left ideal is a two-sided ideal.Left regular bands are precisely the aperiodic $\mathbf{LRBGs}$ and it is known that every $\mathbf{LRBG}$ is a member of $\mathbf{ER}$.  The class $\mathbf{LRBG}$ arises naturally in the study of the  Mantaci-Reutenauer algebra~\cite{MantReut}, which is a wreath product analogue of Solomon's Descent Algebra. As a consequence of the main theorem of the authors' work~\cite{rrbg} we computed the quiver of an algebra associated to the Mantaci-Reutenauer algebra.

Two idempotents $e,f$ in a monoid $M$ are conjugate if there exist elements $x,y \in M$ such that $e=xy,f=yx$. The class $\mathbf{DG}$ consists of all finite monoids such that each conjugacy class of idempotents contains a single element. The $\mathscr{J}$-trivial monoids are precisely the aperiodic monoids in $\mathbf{DG}$. Recall that a finite category $C$ is an $EI$-category if every endomorphism is an isomorphism. The representation theory of $EI$-categories has been extensively studied in recent years. If $C$ is a category, its consolidation $S(C)$ is the semigroup whose elements are the morphisms of $C$ plus a new element 0. The product is that of $C$ when it exists and 0 otherwise. It is easy to see that the algebra of the category $C$ is the algebra of the semigroup $S(C)$ modulo the ideal generated by the element 0 (this is called the contracted algebra of a semigroup with 0).  See~\cite[Chapter~8]{repbook} for details on the connection between monoid representation theory and the representation theory of finite categories. It is easy to check that a skeletal category is an $EI$-category if and only if the semigroup $S(C)$ is in $\mathbf{DG}$. The quiver of algebras of semigroups in $\mathbf{DG}$ is among the quivers in the class of rectangular monoids that is considered in~\cite{DO} by the authors.

A block group is a finite monoid such that every principal right ideal and every principal left ideal has at most one idempotent generator. This is the class of monoids that are both in $\mathbf{ER}$ and in the dual class $\mathbf{EL}$ consisting of monoids such that every principal left ideal has at most one idempotent generator. All inverse monoids are block groups. Examples of block groups that are not inverse monoids include the power monoid $P(G)$ of a finite group $G$. This is the monoid of all subsets of $G$ under multiplication of subsets. Another example is the monoid of all Hall relations $H_n$ on a set of size $n$. $H_n$ consists of all relations $R$ such that $X$ contains a permutation (a perfect matching). Not much is known about the representation theory of general block groups.  Also if $C$ is an $EI$-category, then $S(C)$ is a block group.

A final example of an important monoid in $\mathbf{ER}$ is the following. Let $W$ be a finite Coxeter group and let $F(\mathscr H_W)$ be the hyperplane face monoid associated to the reflection arrangement  $\mathscr H_{W}$ of $W$~\cite{BHR,Brown2}; it is a left regular band.  Since $W$ acts on $F(\mathscr H_W)$ by automorphisms, we can form the semidirect product $F(\mathscr H_W)\rtimes W$, which belongs to $\mathbf{ER}$  Its algebra can be identified with the crossed product of $W$ with $\mathbb CF(\mathscr H_W)$. Saliola has noticed possible applications of this monoid in combinatorics and probability theory. The group $W$ is the group of units of $F(\mathscr H_W)\rtimes W$ and if $e=\frac{1}{|W|}\sum_{w\in W} w$, then it is known that $e\mathbb C[F(\mathscr H_W)\rtimes W]e$ is isomorphic to Solomon's descent algebra $\Sigma(W)$ since the decent algebra is the algebra of invariants $\mathbb CF(\mathscr H_W)^W$ of the hyperplane face monoid algebra (cf,~\cite{Brown2}). Computing the quiver of the crossed product could potentially help in computing the quiver of the descent algebra.

This last example is generic for $\mathbf{ER}$. It is known that $\mathbf{ER}$ consists exactly of the finite monoids that are homomorphic images of a submonoid of a semidirect product of a finite $\mathscr R$-trivial monoid and a finite group~\cite{qtheor}. These examples led us to consider questions related to the representation theory of $\mathbf{ER}$ monoids. One of the main results of this paper computes the quiver of the algebra of a monoid $M$ in $\mathbf{ER}$.

In addition to computing quivers, we consider in this paper the problem of describing the projective indecomposable modules for the algebra of a Fountain monoid.  It is notoriously difficult to write down explicit primitive idempotents for monoids algebras (cf.~\cite{Denton,BergeronSaliola}) and often they have complicated expressions in terms of the monoid basis, making it virtually impossible to determine even the dimension of the corresponding projective indecomposable module let alone construct a matrix representation out of it.  In~\cite{rrbg}, we were able to give an explicit construction of projective indecomposable modules for a family of von Neumann regular monoids, whereas in~\cite{DO} we constructed projective indecomposable modules for $\mathscr R$-trivial monoids as certain partial transformation modules.  This paper  provides the common generalization of these results and gives an explicit description of the projective indecomposable modules for the widest class of monoid algebras to date.

The paper is organized as follows. After a section of preliminaries, we recall the notion of a Fountain monoid and prove some new properties of these monoids under the assumption of finiteness.  Then we turn to the question of describing the projective indecomposable modules for the algebra of a class of Fountain monoid.  The final section uses our construction of the projective indecomposables to compute the quiver of a monoid in $\mathbf{ER}$.  This extends a number of previous results~\cite{Saliola,Jtrivialpaper,DO}.

\section{Preliminaries}

\subsection{Finite monoids}
We recall some basic facts from the theory of finite semigroups and monoids \cite{qtheor}.  The reader is referred to~\cite{CP,qtheor,Arbib,repbook} for details.

Let $M$ be a finite monoid.  If $m\in M$, then $m^{\omega}$ denotes the unique idempotent in the cyclic semigroup generated by $m$.    Note that if $M$ has cardinality $n$, then $m^{\omega}=m^{n!}$ and so $(xy)^{\omega}x=x(yx)^{\omega}$ for all $x,y\in M$.   Green's relations $\mathscr R$, $\mathscr L$ and $\mathscr J$ are defined on $M$ by
\begin{itemize}
\item $m\mathrel{\mathscr L} n$ if $Mm=Mn$;
\item $m\mathrel{\mathscr R} n$ if $mM=nM$;
\item $m\mathrel{\mathscr J} n$ if $MmM=MnM$.
\end{itemize}
The $\mathscr L$-class of $m\in M$ is denoted by $L_m$ and similar notation is used for $\mathscr R$- and $\mathscr J$-classes.  One defines the $\mathscr L$-order on $M$ by $m\leq_{\mathscr L} n$ if $Mm\subseteq Mn$.  The quasi-orders $\leq_{\mathscr R}$ and $\leq_{\mathscr J}$ are defined analogously.  A monoid is called \emph{$\mathscr L$-trivial} if each $\mathscr L$-class is a singleton. One defines analogously $\mathscr R$-trivial monoids and $\mathscr J$-trivial monoids.

The set of idempotents of $M$ is denoted by $E(M)$. An element $m\in M$ is \emph{regular} if $m=mnm$ for some $n\in M$. This is equivalent to $L_m\cap E(M)\neq \emptyset$, $R_m\cap E(M)\neq \emptyset$ and $J_m\cap E(M)\neq \emptyset$ (the last equivalence uses finiteness).  A $\mathscr J$-class is called \emph{regular} if it contains an idempotent or, equivalently, contains only regular elements.  An important fact about finite monoids is that they enjoy a property called \emph{stability} which states that
\begin{align*}
xy\mathrel{\mathscr J} x\iff& xy\mathrel{\mathscr R} x\\
xy\mathrel{\mathscr J} y\iff& xy\mathrel{\mathscr L} y
\end{align*}
for $x,y\in M$~\cite[Theorem~1.13]{repbook}.  One consequence of stability is that any $\mathscr R$-class and $\mathscr L$-class in a $\mathscr J$-class intersect.  Another fact about finite semigroups that we shall use is that if $J$ is a $\mathscr J$-class such that $J^2\cap J\neq \emptyset$, then $J$ is regular (cf.~\cite[Corollary~1.24]{repbook}).

If $e\in E(M)$, then $eMe$ is a monoid with identity $e$ and its group of units is denoted $G_e$ and called the \emph{maximal subgroup} of $M$ at $e$.  For a finite monoid, $G_e=J_e\cap eMe$~\cite[Corollary~1.16]{repbook}.  If $e\mathrel{\mathscr J} f$, then $eMe\cong fMf$ and hence $G_e\cong G_f$~\cite[Corollary~1.12]{repbook}.  The group $G_e$ acts freely on the left of $R_e$ by left multiplication~\cite[Proposition~1.10]{repbook} and two elements belong to the same orbit if and only if they are $\mathscr L$-equivalent~\cite[Corollary~1.17]{repbook}.  Dually, $G_e$ acts freely on the right of $L_e$ by right multiplication and two elements are in the same orbit if and only if they are $\mathscr R$-equivalent.

Elements $x,y\in M$ are (generalized) \emph{inverses} if $xyx=x$ and $yxy=y$.  In this case, $xy,yx$ are $\mathscr J$-equivalent idempotents.  Conversely, if $e,f$ are $\mathscr J$-equivalent idempotents of a finite monoid, then there is an inverse pair $x,y$ with $xy=e$ and $yx=f$.  Then $z\mapsto yzx$ gives an isomorphism of $eMe$ with $fMf$, which restricts to an isomorphism of $G_e$ with $G_f$.

A monoid in which every element has at most one inverse is called a \emph{block group}.  Classical examples of block groups are inverse monoids (monoids in which each element has exactly one inverse; these have semisimple algebras in good characteristic), monoids with commuting idempotents, $\mathscr J$-trivial monoids (see~\cite{Jtrivialpaper} for interesting examples from the point of view of representation theory), power sets of finite groups and the monoid of Hall relations (a Hall relation is a binary relation containing a perfect matching). It is a deep theorem of finite semigroup theory that every finite block group is a quotient of a subsemigroup of the power set of a finite group; see~\cite[Chapter~4]{qtheor} for details.

A finite monoid $M$ is called \emph{aperiodic} if all its maximal subgroups are trivial; this is equivalent to there existing $k>0$ with $m^k=m^{k+1}$ for all $m\in M$.

The class $\mathbf{ER}$ consists of those finite monoids $M$ whose idempotents generate an $\mathscr R$-trivial monoid.  The classes $\mathbf{EL}$ and $\mathbf {EJ}$ are defined similarly.  It is known that $\mathbf{EJ}=\mathbf{ER}\cap \mathbf{EL}$ is the class of finite block groups~\cite{qtheor,Almeida:book}.  Also, a finite monoid $M$ belongs to $\mathbf{ER}$ if and only if it contains no two-element right zero subsemigroup, that is, $M$ does not contain two $\mathscr R$-equivalent idempotents.  Every monoid in $\mathbf{ER}$ has the property that it has a unique minimal left ideal, which is also its unique minimal two-sided ideal.  This follows from basic structural properties of minimal ideals of finite semigroups~\cite[Appendix~A]{qtheor}.  Namely,  the two-sided minimal ideal is the disjoint union of all minimal left ideals and each minimal left ideal  can be generated by an idempotent.  If $e,f$ are idempotents generating minimal left ideals, then $(ef)^{\omega}$ must generate the same minimal left ideal as $f$ and so without loss of generality, we may assume that $ef=f$.  But, by stability, it then follows that $f$ generates the same principal right ideal as $e$ and so $fe=e$.  As the idempotents generate an $\mathscr R$-trivial monoid, we deduce that $e=f$.  Note that the class $\mathbf{ER}$ is closed under direct product, submonoids and quotient monoids.  The canonical example of a monoid in $\mathbf{ER}$ is a semidirect product $R\rtimes G$ with $G$ a finite group acting on a finite $\mathscr R$-trivial monoid $R$ by automorphisms. In fact, a result of Stiffler~\cite{Stiffler,qtheor} implies that each finite monoid in $\mathbf{ER}$ is a quotient of a subsemigroup of such a semidirect product.

If $J$ is a regular $\mathscr J$-class and $e\in E(J)$, then a \emph{sandwich matrix} for $J$ is a matrix over $G_e\cup \{0\}$ obtained in the following way.  Let $A$ be the set of $\mathscr R$-classes of $J$ and $B$ be the set of $\mathscr L$-classes of $J$.  By the elementary structure theory of finite semigroups (cf.~\cite[Appendix~A, page 600]{qtheor} or ~\cite{Arbib}) each $\mathscr L$-class of $J$ meets $R_e$ and each $\mathscr R$-class of $J$ meets $L_e$.  Choose representatives $\rho_a\in L_e\cap a$ for each $a\in A$ and $\lambda_b\in R_e\cap b$ for each $b\in B$.  Then $\lambda_b\rho_a\in eMe$ for each $a\in A$ and $b\in B$.  Define a matrix $P\colon B\times A\to G_e\cup \{0\}$ by
\[P(b,a)= \begin{cases} \lambda_b\rho_a, & \text{if}\ \lambda_b\rho_a\in G_e\\ 0, & \text{else.}  \end{cases}\]  One can show that if one changes the representatives of the $\mathscr L$-classes and $\mathscr R$-classes, then the sandwich matrix will change by left and right multiplication by diagonal matrices over $G_e$.  Hence if $\Bbbk$ is a commutative ring with unit, then properties like left, right or two-sided invertibility of $P$ over $\Bbbk G_e$ does not depend on the choices.  One can also show that these properties do not depend on the choice of the idempotent $e$~\cite{Arbib}.

Sometimes, it is fruitful to view $P$ as follows.  We have that $\mathbb ZL_e$ is a free right $\mathbb ZG_e$-module on $|A|$ generators and $\mathbb ZR_e$ is a free left $\mathbb ZG_e$-module on $|B|$ generators.  Thus $\Hom_{\mathbb ZG_e}(\mathbb ZR_e,\mathbb ZG_e)$ is a free right $\mathbb ZG_e$-module on $|B|$ generators.  There is a natural right $\mathbb ZG_e$-module homomorphism $T\colon \mathbb ZL_e\to \Hom_{\mathbb ZG_e}(\mathbb ZR_e,\mathbb ZG_e)$ given by \[T(\ell)(r) = \begin{cases} r\ell, & \text{if}\ r\ell\in G_e\\ 0, &\text{else}\end{cases}\] for $\ell\in L_e$ and $r\in R_e$.  It is easy to check that $P$ is the matrix of $T$ with respect to an appropriate choice of bases.

If $M$ is a block group, then $P$ can always be taken to be an identity matrix and if $M\in \mathbf{ER}$, then $P$ can be taken to have a block diagonal form, where each diagonal block is a row of identity elements of $G_e$.  Hence, for any monoid $M\in \mathbf{ER}$, the sandwich matrices are right invertible over $\mathbb Z$.  We do not prove these assertions here, but we will prove later that $T$ is surjective.

\subsection{Finite dimensional algebras}
Next we review some basic elements of the theory of finite dimensional algebras.  References for this material include~\cite{assem,curtis,benson,Isaacs}.
Fix a field $\Bbbk$. Let $A$ be a finite dimensional $\Bbbk$-algebra.  The \emph{radical} $\rad(A)$ of $A$ is its largest nilpotent ideal.  It is also the intersection of all maximal left (right) ideals. It is the smallest ideal $I$ such that $A/I$ is semisimple. If $V$ is any finite dimensional (left) $A$-module, then the \emph{radical} $\rad(V)$ is the intersection of all maximal submodules of $V$.  One has that $\rad(V)=\rad(A)V$ and $\rad(V)$ is the smallest submodule such that $V/\rad(V)$ is semisimple.  One usually calls $V/\rad(V)$ the \emph{top} of $V$.    The \emph{socle} $\soc{V}$ of $V$ is its largest semisimple submodule, that is, the submodule generated by all simple submodules of $V$.

If $V$ is a right $A$-module, then the vector space dual $D(V)=\Hom_{\Bbbk}(V,\Bbbk)$ is a left $A$-module via $(af)(v) = f(va)$ for $f\colon V\to \Bbbk$, $v\in V$ and $a\in A$ (and dually, one can go from left modules to right modules and we use the same notation).

A finite dimensional algebra $A$ is said to be \emph{split} if $A/\rad(A)$ is a direct product of matrix algebras over $\Bbbk$.  For example, any algebra over an algebraically closed field is split.  This is equivalent to the endomorphism monoid of each simple $A$-module being isomorphic to $\Bbbk$.
 We recall that if $A$ is a semisimple algebra and $S_1,\ldots, S_k$ are representatives of the isomorphism classes of simple $A$-modules, then \[A\cong \bigoplus_{i=1}^k \frac{\dim S_i}{\dim \End_A(S_i)}\cdot S_i\] as an $A$-module.  Also note that every module over a semisimple algebra $A$ is projective and hence every left ideal of $A$ is of the form $Ae$ with $e$ an idempotent.

If $G$ is a group, then $\Bbbk$ is called a \emph{splitting field} for $G$ if $\Bbbk G$ is split.  A famous result of Brauer~\cite[Corollary~9.15, Theorem~10.3]{Isaacs} asserts that if $G$ is of exponent $n$ and $\Bbbk$ contains a primitive $n^{th}$-root of unity, then $\Bbbk$ is a splitting field for $G$.  It is also known that, for a monoid $M$, the algebra $\Bbbk M$ is split if and only if $\Bbbk$ is a splitting field for all the maximal subgroups of $M$, cf.~Proposition~\ref{p:endos} below.
We recall that a group algebra $\Bbbk G$ is semisimple if and only if the characteristic of $\Bbbk$ does not divide $|G|$ by Maschke's theorem.

A module $V$ is \emph{indecomposable} if it cannot be expressed as a direct sum of proper submodules.  If $V$ is a module with $V/\rad(V)$ simple, then necessarily $V$ is indecomposable.  By the Krull-Schmidt theorem, each finite dimensional $A$-module $V$ can be written as a direct sum of indecomposable modules and the isomorphism classes (with multiplicities) are unique. In particular, we can decompose the regular $A$-module as $A=P_1\oplus\cdots \oplus P_s$ where the $P_i$ are projective indecomposable modules.  One has that $P_i/\rad(P_i)$ is simple and $P_i\cong P_j$ if and only if $P_i/\rad(P_i)\cong P_j/\rad(P_j)$.  Moreover, every simple $A$-module is isomorphic to one of the form $P_i/\rad(P_i)$ and every projective indecomposable module is isomorphic to some $P_i$.  Let us assume that $P_1,\ldots, P_k$ form a complete set of representatives of the isomorphism classes of projective indecomposable modules and let $S_i=P_i/\rad(P_i)$ be the corresponding simple module.  Then \[A\cong \bigoplus_{i=1}^k \frac{\dim S_i}{\dim \End_A(S_i)}\cdot P_i.\]  Each finite dimensional $A$-module $V$ has a \emph{projective cover} $P$.  This is a finite dimensional projective module $P$ with a surjective homomorphism $\psi\colon P\to V$ such that $\ker\psi\subseteq \rad(P)$.  The projective module $P$ is unique up to isomorphism and satisfies $P/\rad(P)\cong V/\rad(V)$.  If $S$ is a simple module, its projective cover is the unique projective indecomposable module $P$ with $P/\rad(P)\cong S$.

The injective indecomposable $A$-modules are exactly the vector space duals of projective indecomposable right $A$-modules.  If $S$ is a simple module, there is a unique (up to isomorphism) injective indecomposable module $I$ whose socle is isomorphic to $S$ called the \emph{injective envelope} of $S$; it is the vector space dual of the right projective cover of $D(S)$, which is a simple right $A$-module.

A finite dimensional algebra $A$ is said to be \emph{split basic} if each simple module is one-dimensional.  Every split $\Bbbk$-algebra is Morita equivalent to a unique (up to isomorphism) basic one.  Moreover, Gabriel described split basic algebras in terms of quivers.
A \emph{quiver} $Q$ is  a finite directed graph (possibly with loops and multiple edges).  The path algebra $\Bbbk Q$ is the $\Bbbk$-algebra with basis the directed paths in $Q$ with product induced by concatenation (where undefined concatenations are made $0$).  We allow an empty path at each vertex and compose from right to left as in  a category.   Let $J$ be the ideal spanned by all non-empty paths.  An ideal $I$ of $\Bbbk Q$ is called \emph{admissible} if $J^n\subseteq I\subseteq J^2$ for some $n\geq 2$.  If $I$ is admissible, then $\Bbbk Q/I$ is a finite dimensional split basic $\Bbbk$-algebra and every split basic finite dimensional $\Bbbk$-algebra is isomorphic to one of this form by a theorem of Gabriel.  The quiver $Q$ is unique up to isomorphism, but the ideal $I$ is not.

The \emph{quiver} of a split algebra $A$ is the directed graph with vertex set the set of isomorphism classes $[S]$ of simple $A$-modules $S$.  The number of directed edges from $[S]$ to $[S']$ is $\dim \Ext^1_A(S,S')$.  If $P$ is the projective cover of $S$, then using the long exact sequence for $\Ext$, one easily checks that $\Ext^1_A(S,S')\cong \Hom_A(\rad{P},S')\cong \Hom_A(\rad(P)/\rad^2(P),S')$.  The basic algebra Morita equivalent to $A$ is isomorphic to the path algebra on its quiver modulo an admissible ideal.  Thus computing the quiver of a finite dimensional algebra, like a monoid algebra, is the first step toward understanding its representation theory.

\subsection{The representation theory of finite monoids}\label{ss:reptheory}
We review here key aspects of the representation theory of finite monoids.   More details can be found in~\cite{repbook}.
Let $\Bbbk$ be a field and $M$ a finite monoid.  Fix $e\in E(M)$.  Then $\Bbbk L_{e}$ is a $\Bbbk M$-$\Bbbk G_{e}$-bimodule, where $M$ acts on the left of $L_{e}$ via left multiplication if the result is in $L_{e}$, and otherwise the result is $0$.  Similarly, $\Bbbk R_{e}$ is a $\Bbbk G_{e}$-$\Bbbk M$-bimodule.
 If $V$ is a $\Bbbk G_{e}$-module, we put
 \begin{align*}
 \Ind_{G_{e}}(V)&=\Bbbk L_e\otimes_{\Bbbk G_e} V\\ \Coind_{G_e}(V)&=\Hom_{\Bbbk G_e}(\Bbbk R_e,V)\cong \Hom_{G_e}(R_e,V)\cong D(D(V)\otimes_{\Bbbk G_e}\Bbbk R_e)
  \end{align*}
  where $\Hom_{G_e}(R_e,V)$ is the set of $G_e$-equivariant mappings $R_e\to V$.  One has natural isomorphisms $e\Ind_{G_e}(V)\cong V\cong e\Coind_{G_e}(V)$ and there is a unique $\Bbbk M$-module homomorphism $\p_V\colon \Ind_{G_e}(V)\to \Coind_{G_e}(V)$ which extends the natural isomorphism $e\Ind_{G_e}(V)\to e\Coind_{G_e}(V)$.  If $V$ is semisimple, then $\p_V$ is injective if and only if $\Ind_{G_e}(V)$ is semisimple and $\p_V$ is surjective if and only if $\Coind_{G_e}(V)$ is semisimple (cf.~\cite[Corollary~4.22]{repbook}).

 If $V$ is simple, then  $\Ind_{G_e}(V)$ and $\Coind_{G_e}(V)$ are indecomposable modules and $V^\sharp=\soc{\Coind_{G_e}(V)}=\Bbbk Me\Coind_{G_e}(V)$ is simple and is the image of $\p_V$; hence $\ker \p_V=\rad(\Ind_{G_e}(V))$ and $V^\sharp\cong \Ind_{G_e}(V)/\rad(\Ind_{G_e}(V))$.  The natural morphism $\p_V$ is induced by the sandwich matrix of the $\mathscr J$-class of $e$.  If the characteristic of $\Bbbk$ does not divide the order of $G_e$, then one has that the sandwich matrix of $J_e$ is right invertible over $\Bbbk G_e$ if and only if $\Coind_{G_e}(V)=V^\sharp$ for all simple $\Bbbk G_e$-modules $V$ (cf.~\cite[Lemma~5.20]{repbook} and~\cite[Corollary~4.22]{repbook}).

Fix idempotent representatives $e_1,\ldots, e_n$ of the regular $\mathscr J$-classes of $M$.  Then the isomorphism classes of simple $\Bbbk M$-modules are in bijection with pairs $(i,[V])$ where $i=1,\ldots, n$ and $V$ is a simple $\Bbbk G_{e_i}$-module.  The corresponding simple $\Bbbk M$-module is $V^\sharp$.

\section{Fountain monoids}
The equivalence relation $\til {\mathscr L}$ is defined on a monoid $M$ by $m\mathrel{\til {\mathscr L}} n$ if, for all idempotents $e\in E(M)$, we have $me=m$ if and only if $ne=n$~\cite{Ltilde}.  The $\til {\mathscr L}$-class of $m$ is denoted $\til L_m$.  The relation $\til {\mathscr R}$ is defined dually and we use $\til R_m$ for the $\til{\mathscr R}$-class of $m$.  Notice that if $e$ is an  idempotent, then since $ee=e$ if $m\in \til L_e$, then $me=m$ and so $m\in Me$.  If $m\in Me$, then, for an idempotent $f$, we have that $ef=e$ implies $mf=mef=me=m$ and so $m\in \til L_e$ if and only if $mf=m$ implies $ef=e$ for each idempotent $f\in E(M)$ (cf.~\cite[Lemma~2.1]{Ltilde}).  It follows that $Me\setminus \til L_e$ is a left ideal and so $M$ acts on the left of $\til L_e$ by partial mappings.  That is, if $x\in \til L_e$ and $m\in M$, then the action of $m$ on $x$ is defined if and only if $mx\in \til L_e$, in which case the result of the action is $mx$.  Also it is easy to see that if $m,n\in M$ are regular, then $m\mathrel{\mathscr L} n$ if and only if $m\mathrel{\til{\mathscr{L}}} n$~\cite[Proposition~17.3]{repbook}.    Thus if $e\in E(M)$, then $\til L_e\setminus L_e$ consists of non-regular elements.

If $M$ is a monoid and $m\in M$, then the \emph{right stabilizer} of $m$ is the submonoid of $M$ consisting of those elements $n\in M$ with $mn=m$.

\begin{Prop}\label{p:char.til}
Let $M$ be a finite monoid and $m\in M$. Then $\til L_m$ contains an idempotent if and only if the right stabilizer of $m$ has a unique minimal left ideal.
\end{Prop}
\begin{proof}
Let $N$ be the right stabilizer of $m$.  Suppose that $e\in \til L_m$ is an idempotent.  Then since $ee=e$, we have that $me=m$ and so $e\in N$.  We claim that $Ne$ is the unique minimal left ideal of $N$.  Indeed, if $n\in N$, then $mn^{\omega}=m$ and so $en^{\omega}=e$.  Thus $e\in Nn$.

Conversely, suppose that $N$ has a unique minimal left ideal $L$ and let $e\in E(L)$. Note that $me=m$ since $e\in N$. We claim that $e\in \til L_m$.  By the discussion above we must show that if $f\in E(M)$ and $mf=m$, then $ef=e$.  But then $f\in N$ and so $L\subseteq Nf$, whence $ef=e$.  Thus $e\in \til L_m$, thereby completing the proof.
\end{proof}

A monoid is called \emph{right Fountain} if each $\til{\mathscr L}$-class contains an idempotent.  Left Foutain monoids are defined dually.  A monoid is \emph{Fountain} if it is both left and right Fountain. In the literature, the term ``semi-abundant'' is used wherever we have used Fountain, but we have renamed the class in John Fountain's honor as he promoted the study of these monoids over the years.

\begin{Cor}\label{c:examples}
Let $M$ be a finite monoid.  If $M\in \mathbf{ER}$ (i.e., contains no two-element right zero subsemigroups) or if $M$ is a regular monoid, then $M$ is right Fountain.  In particular, block groups (e.g., $\mathscr J$-trivial monoids) are Fountain and $\mathscr R$-trivial monoids (or monoids where regular $\mathscr J$-classes are $\mathscr L$-classes) are right Fountain.
\end{Cor}

Although $\til{\mathscr L}$ is not in general a right congruence, it enjoys the following congruence-like property in finite right Fountain monoids.

\begin{Prop}\label{p:local.works}
Let $M$ be a finite right Fountain monoid and suppose that $e,f\in E(M)$ are $\mathscr J$-equivalent.  If $x,y\in M$ with $xyx=x$, $yxy=y$, $xy=e$ and $yx=f$, then $\rho\colon \til L_e\to \til L_f$ given by $\rho(m)=mx$ is a bijection with inverse $\psi\colon \til L_f\to \til L_e$ given by  $\psi(n)=ny$.
\end{Prop}
\begin{proof}
First we must show that if $m\in \til L_e$, then $mx\in \til L_f$.  Let $a\in E(M)$ be an idempotent in $\til L_{mx}$, and so $mxa=mx$.  We show that $a\mathrel{\mathscr L} f$ and hence $mx\in \til L_f$.    Since $mxf=mx$, we have that $af=a$ and hence $fafa=faa=fa$.  Notice that $mxay=mxy=me=m$ and so $m(xay)^{\omega}=m$.  Therefore, $(xay)^{\omega}=xy(xay)^{\omega}=e(xay)^{\omega}=e$ as $m\mathrel{\til {\mathscr L}} e$.  Thus, $f=yex=y(xay)^{\omega}x= (yxa)^{\omega}yx=(fa)^{\omega}f=faf=fa$ and so $a\mathrel{\mathscr L} f$.  Therefore, we obtain  $mx\in \til L_f$.

A dual argument shows that $\psi$ maps $\til L_f$ to $\til L_e$.  Clearly, if $m\in \til L_e\subseteq Me$, then $\psi(\rho(m)) = mxy=me=m$ and dually, $\rho\psi$ is the identity.
\end{proof}

We state two immediate consequences of the proposition.

\begin{Cor}\label{c:have.action}
If $M$ is a finite right Fountain monoid and $e\in E(M)$, then $G_e$ acts on the right of $\til L_e$ by right multiplication.
\end{Cor}
\begin{proof}
If $g\in G_e$ with inverse $g'\in G_e$, then $gg'g=e=g'gg'$ and so right multiplication by $g$ yields a bijection from $\til L_e$ to $\til L_e$ with inverse $g'$.
\end{proof}

\begin{Cor}\label{c:partial.trans}
Let $M$ be a finite right Fountain monoid and $\Bbbk$ a field.  Then, for $e\in E(M)$, the partial transformation module $\Bbbk \til L_e$ is a $\Bbbk M$-$\Bbbk G_e$-bimodule.  If $f\in E(M)$ with $e\mathrel{\mathscr J} f$, then $\Bbbk \til L_e\cong \Bbbk \til L_f$ as left $\Bbbk M$-modules.
\end{Cor}
\begin{proof}
The first statement follows from Corollary~\ref{c:have.action}.  For the second item, since $e\mathrel{\mathscr J} f$ there exist $x,y\in M$ with $xyx=x$, $yxy=y$, $xy=e$ and $yx=f$.  Then $\rho\colon \til L_e\to \til L_f$ given by $\rho(m)=mx$ induces a $\Bbbk M$-module isomorphism $\Bbbk \til L_e\to \Bbbk \til L_f$ by Proposition~\ref{p:local.works}.
\end{proof}

\section{Projective indecomposable modules}
We give an explicit construction of the projective indecomposable modules for a natural class of right Fountain monoids.  This family includes all $\mathscr R$-trivial monoids and all regular monoids whose sandwich matrices are right invertible, and hence includes all families of monoids for which we have previously constructed projective indecomposable modules~\cite{rrbg,DO}.

\begin{Prop}\label{p:radical}
Let $M$ be a finite monoid and $\Bbbk$ a field.  If $L$ is a left ideal of $M$, then $\Bbbk L+\rad(\Bbbk M)=\Bbbk L^2+\rad(\Bbbk  M)$.
\end{Prop}
\begin{proof}
The inclusion from right to left is obvious.  For the other direction, let $A=\Bbbk M/\rad(\Bbbk M)$.  Then $(\Bbbk L+\rad(\Bbbk M))/\rad(\Bbbk M)$ is a left ideal of the semisimple algebra $A$ and hence is generated as a left ideal by an idempotent $e=\sum_{m\in L}c_mm+\rad(\Bbbk M)$.  Then \[e=e^2=\sum_{m,n\in L}c_mc_nmn+\rad(\Bbbk M)\in (\Bbbk L^2+\rad(\Bbbk M))/\rad(\Bbbk M).\]  This establishes the inclusion from left to right.
\end{proof}

The following simple lemma will be useful to prove our main result.

\begin{Lemma}\label{l:rad.hit}
Let $A$ and $B$ be finite dimensional algebras and let $V$ be a finite dimensional $A$-$B$-bimodule.   Let $e\in B$ be an idempotent.  Then $\rad(Ve)=\rad(V)e$ (where the radical is taken as $A$-modules).
\end{Lemma}
\begin{proof}
We have $\rad(Ve)=\rad(A)(Ve)=(\rad(A)V)e=\rad(V)e$, as required.
\end{proof}

We write $x<_{\mathscr L} y$ to indicate that $Mx\subsetneq My$.

\begin{Prop}\label{p:in.radical}
Let $M$ be a finite monoid and $\Bbbk$ a field.  If $e\in E(M)$, then $\Bbbk [\til L_e\setminus L_e]\subseteq \rad(\Bbbk \til L_e)$.
\end{Prop}
\begin{proof}
If $\til L_e=L_e$, there is nothing to prove and so we assume that $L_e\subsetneq \til L_e$.
Let $x\in \til L_e\setminus L_e$.  Assume that each $y\in \til L_e\setminus L_e$ with $y<_{\mathscr L} x$ belongs to $\rad(\Bbbk \til L_e)$. (This holds vacuously if $x$ is $<_{\mathscr L}$-minimal in $\til L_e\setminus L_e$.)  We show that $x\in \rad(\Bbbk \til L_e)$.  Let $L=Mx$.  Then by Proposition~\ref{p:radical} we have that $x=\sum_{i\in J}c_ia_ixb_ix+r$ with $c_i\in \Bbbk$, $a_i,b_i\in M$ and $r\in \rad(\Bbbk M)$.  Let $F\subseteq J$ be the set of indices $i$ with $a_ixb_ix\in \til L_e$.  Notice that since $x$ is not regular, $a_ixb_ix<_{\mathscr L} x$ (and hence $a_ixb_ix\notin L_e$ as $x\in Me$). Thus $a_ixb_ix\in \rad(\Bbbk \til L_e)$ whenever $i\in F$ by our assumption. Therefore, in $\Bbbk \til L_e$, we have that
\[x=xe=\sum_{i\in F} c_ia_ixb_ix+re\in \rad(\Bbbk \til L_e).\]  This completes the proof.
\end{proof}

The following proposition is an elementary exercise in representation theory.

\begin{Prop}\label{p:rad.quotient}
If $U$ and $V$ are finite dimensional modules over a finite dimensional algebra $A$ with $U\subseteq \rad(V)$, then $\rad(V/U)=\rad(V)/U$ and $V/\rad(V)\cong (V/U)/\rad(V/U)$.
\end{Prop}
\begin{proof}
We compute that $\rad(V/U)=\rad(A)\cdot V/U = (\rad(A)V+U)/U=(\rad(V)+U)/U=\rad(V)/U$.  The second statement follows from the first and the usual isomorphism theorem.
\end{proof}

We recall that if $M$ is right Fountain, then $\Bbbk \til L_e$ is a $\Bbbk M$-$\Bbbk G_e$-bimodule for $e\in E(M)$ by Corollary~\ref{c:partial.trans}.

\begin{Cor}\label{C:is.indec}
Let $M$ be a finite right Fountain monoid, $\Bbbk$ a field and $e\in E(M)$.  Suppose that the characteristic of $\Bbbk$ does not divide $|G_e|$ and that $V$ is a simple $\Bbbk G_e$-module.  Then $\Bbbk \til L_e\otimes_{\Bbbk G_e} V$ is an indecomposable $\Bbbk M$-module with simple top $V^\sharp$ (the simple $\Bbbk M$-module corresponding to $V$).
\end{Cor}
\begin{proof}
Let $\eta$ be a primitive idempotent of $\Bbbk G_e$ with $V\cong \Bbbk G_e\eta$.  Consider the exact sequence of $\Bbbk M$-$\Bbbk G_e$-bimodules
\[0\longrightarrow \Bbbk [\til L_e\setminus L_e]\longrightarrow \Bbbk \til L_e\longrightarrow \Bbbk L_e\longrightarrow 0.\]  Then since $V$ is a projective $\Bbbk G_e$-module, we obtain an exact sequence of $\Bbbk M$-modules
\[0\longrightarrow \Bbbk [\til L_e\setminus L_e]\otimes_{\Bbbk G_e}V\longrightarrow \Bbbk \til L_e\otimes_{\Bbbk G_e}V\longrightarrow \Bbbk L_e\otimes_{\Bbbk G_e} V\longrightarrow 0.\]  As $\Bbbk L_e\otimes_{\Bbbk G_e} V=\Ind_{G_e}(V)$ is indecomposable with simple top $V^\sharp$ by the discussion in Subsection~\ref{ss:reptheory},  it suffices, by Proposition~\ref{p:rad.quotient}, to show that  $\Bbbk [\til L_e\setminus L_e]\otimes_{\Bbbk G_e}V$ is contained in the radical of $\Bbbk \til L_e\otimes_{\Bbbk G_e} V$.  But if $U$ is an $A$-$\Bbbk G_e$-bimodule, with $A$ a finite dimensional $\Bbbk$-algebra, then one has that $U\otimes_{\Bbbk G_e} V\cong U\eta$ as an $A$-module and so the result follows from Proposition~\ref{p:in.radical} and Lemma~\ref{l:rad.hit}.
\end{proof}

The next proposition shows that if $V$ is a simple $\Bbbk G_e$-module and $V^\sharp$ is the corresponding simple $\Bbbk M$, then both of these modules have isomorphic endomorphism algebras.

\begin{Prop}\label{p:endos}
Let $M$ be a finite monoid and $e\in E(M)$.  
Let $V$ be a simple $\Bbbk G_e$-module.  Then there is an isomorphism $\End_{\Bbbk G_e}(V)\cong \End_{\Bbbk M}(V^\sharp)$ where $V^\sharp =\soc{\Coind_{G_e}(V)}$ is the corresponding simple $\Bbbk M$-module.
\end{Prop}
\begin{proof}
Recall that $V\cong eV^\sharp$ as a $\Bbbk G_e$-module; we shall work with the latter.  If $\p\colon V^\sharp\to V^\sharp$ is an endomorphism, then $\p|_{eV^\sharp}\colon eV^\sharp\to eV^\sharp$ is a $\Bbbk G_e$-module endomorphism. Thus there is a restriction homomorphism $T\colon \End_{\Bbbk M}(V^\sharp)\to \End_{\Bbbk G_e}(eV^\sharp)$.  Moreover, $T$ is injective because $\End_{\Bbbk M}(V^\sharp)$ is a division algebra by Schur's lemma.  On the other hand, there is an isomorphism $\End_{\Bbbk M}(\Coind_{G_e}(V))\to \End_{\Bbbk G_e}(V)$ obtained by restricting $\p\in \End_{\Bbbk M}(\Coind_{G_e})$ to $e\Coind_{G_e}(V) = e\soc{\Coind_{G_e}(V)}=eV^{\sharp}$ (see~\cite[Proposition~4.6]{repbook}).  Moreover, $\p(V^\sharp)=\p(\soc{\Coind_{G_e}(V)})\subseteq \soc{\Coind_{G_e}(V)}=V^\sharp$ and so we conclude that $T$ is surjective.  This completes the proof.
\end{proof}

It follows from Proposition~\ref{p:endos} that $\Bbbk$ is a splitting field for $M$ if and only if it is for each maximal subgroup of $M$, as was asserted earlier.

We now prove the main theorem of this section. It simultaneously generalizes our previous results for $\mathscr R$-trivial monoids~\cite{DO} and regular monoids~\cite{rrbg}.

\begin{Thm}\label{t:projindec}
Let $M$ be a finite right Fountain monoid and $\Bbbk$ a field whose characteristic does not divide the order of any maximal subgroup of $M$. Let $e_1,\ldots, e_n$ be a complete set of idempotent representatives of the regular $\mathscr J$-classes of $M$.  Then the following are equivalent.
\begin{enumerate}
\item Each sandwich matrix of $J_{e_i}$ is right invertible over $\Bbbk G_{e_i}$, for $i=1,\ldots, n$.
\item Each coinduced module $\Coind_{G_{e_i}}(V)$ with $V$ a simple $\Bbbk G_{e_i}$-module is simple, for $i=1,\ldots, n$.
\item The projective cover of the simple module $V^\sharp$ associated to each simple $\Bbbk G_{e_i}$-module $V$ is $\Bbbk \til L_{e_i}\otimes_{\Bbbk G_{e_i}} V$.
\end{enumerate}
\end{Thm}
\begin{proof}
First of all, the assumption that the sandwich matrix of $J_{e_i}$ is right invertible is equivalent to the assertion that $V^\sharp=\Coind_{G_{e_i}}(V)$ for each simple $\Bbbk G_{e_i}$-module $V$ by~\cite[Corollary~4.22]{repbook} and~\cite[Lemma~5.20]{repbook}.

Assume first that (2) holds.  Then
 Corollary~\ref{C:is.indec} shows that $\Bbbk \til L_{e_i}\otimes_{\Bbbk G_{e_i}} V$ is an indecomposable module with simple top $\Coind_{G_{e_i}}(V)$.

 Note that  $\Coind_{G_{e_i}}(V)=\Hom_{G_{e_i}}(R_{e_i},V)$ is isomorphic as a vector space to $V^{\ell_i}$ where $\ell_i$ is the number of $\mathscr L$-classes in $J_{e_i}$ because $R_{e_i}$ is a free left $G_{e_i}$-set with basis of cardinality $\ell_i$.  Since $\End_{\Bbbk M}(\Coind_{G_{e_i}}(V))\cong \End_{\Bbbk G_{e_i}}(V)$ (cf.~\cite[Proposition~4.6]{repbook}), it follows that
\[\Bbbk M/\rad(\Bbbk M) = \bigoplus_{i=1}^n\bigoplus_{V\in \mathrm{Irr}_{\Bbbk}(G_{e_i})} \frac{\ell_i\dim V}{\dim \End_{\Bbbk G_{e_i}}(V)}\cdot \Coind_{G_{e_i}}(V).\]

On the other hand, if $e\in E(M)$, then
\begin{equation}\label{eq:help}
\begin{split}
\Bbbk \til L_e&\cong \Bbbk \til L_e\otimes_{\Bbbk G_e}\Bbbk G_e\\ & \cong \Bbbk \til L_e\otimes_{\Bbbk G_e}\left(\bigoplus_{V\in \mathrm{Irr}_{\Bbbk}(G_e)} \frac{\dim V}{\dim \End_{\Bbbk G_e}(V)}\cdot V\right)\\ &\cong \bigoplus_{V\in \mathrm{Irr}_{\Bbbk}(G_e)}\frac{\dim V}{\dim \End_{\Bbbk G_e}(V)}\cdot (\Bbbk \til L_e\otimes_{\Bbbk G_e} V).
\end{split}
\end{equation}
  Also, if $e\mathrel{\mathscr J} f$, then $\Bbbk \til L_e\cong \Bbbk \til L_f$ by Corollary~\ref{c:partial.trans}.  Therefore, if $U=\bigoplus_{e\in E(M)/\mathscr{L}} \Bbbk \til L_e$, then $U\cong \bigoplus_{i=1}^n \ell_i\cdot \Bbbk \til L_{e_i}$ and so
\begin{align*}
U/\rad(U)&\cong \bigoplus_{i=1}^n\ell_i\cdot \left(\bigoplus_{V\in \mathrm{Irr}_{\Bbbk}(G_{e_i})} \frac{\dim V}{\dim \End_{\Bbbk G_{e_i}}(V)}\cdot \Coind_{G_{e_i}}(V)\right)\\ &\cong \Bbbk M/\rad(\Bbbk M).
\end{align*}
  Thus there is a projective cover $\psi\colon \Bbbk M\to U$.  But since $\til{\mathscr L}$ is an equivalence relation and each $\til{\mathscr L}$-class contains a unique $\mathscr L$-class of idempotents, we deduce that $\dim U=|M|$ and so $\psi$ is an isomorphism.  Therefore, $U$ is a projective module and hence each of its direct summands $\Bbbk \til L_e\otimes_{\Bbbk G_e}V$ is a projective module.   This completes the proof that  (2) implies (3).

Next assume that (3) holds.  Again put \[U=\bigoplus_{e\in E(M)/\mathscr{L}} \Bbbk \til L_e\cong \bigoplus_{i=1}^n \ell_i\cdot \Bbbk \til L_{e_i}.\]  Then, as above, we have that $\dim U=|M|$ since $M$ is right Fountain.  Also $U$ is projective by hypothesis.  As $V^\sharp$ is a submodule of $\Coind_{G_{e_i}}(V)$ for a simple $\Bbbk G_{e_i}$-module $V$, we have $\dim V^\sharp\leq \ell_i\cdot \dim V$. It follows from \eqref{eq:help} and Corollary~\ref{C:is.indec} that \[U/\rad(U)\cong \bigoplus_{i=1}^n \bigoplus_{V\in \Irr_{\Bbbk}(G_{e_i})} \ell_i\frac{\dim V}{\dim \End_{\Bbbk G_{e_i}}(V)}\cdot V^\sharp\] and so $U$ contains $\ell_i\dim V/\dim \End_{\Bbbk G_{e_i}}(V)$ copies of the projective cover of $V^\sharp$ in its decomposition into indecomposable modules.  On the  other hand, $\Bbbk M$ has $\dim V^\sharp/\dim \End_{\Bbbk M}(V^\sharp)$ copies of the projective cover of $V^\sharp$ in its decomposition into indecomposable modules.  From the equalities $\dim \End_{\Bbbk M}(V^\sharp)=\dim \End_{\Bbbk G_{e_i}}(V)$ (from Proposition~\ref{p:endos}) and  $\dim \Bbbk M=\dim U$, we conclude that $\dim V^\sharp=\ell_i\cdot \dim V$ for all $i$ and $V$ and hence $\Coind_{G_{e_i}}(V)$ is simple for all $i$ and $V$, establishing (2).
This completes the proof.
\end{proof}

Let us state the dual to Theorem~\ref{t:projindec}.

\begin{Thm}\label{t:inj.indec}
Let $M$ be a finite left Fountain monoid and $\Bbbk$ a field whose characteristic does not divide the order of any maximal subgroup of $M$. Let $e_1,\ldots, e_n$ be a complete set of idempotent representatives of the regular $\mathscr J$-classes of $M$.  Then the following are equivalent.
\begin{enumerate}
\item Each sandwich matrix of $J_{e_i}$ is left invertible over $\Bbbk G_{e_i}$, for $i=1,\ldots, n$.
\item Each induced module $\Ind_{G_{e_i}}(V)$ with $V$ a simple $\Bbbk G_{e_i}$-module is simple, for $i=1,\ldots, n$.
\item The injective envelope of the simple module $V^\sharp$ associated to each simple $\Bbbk G_{e_i}$-module $V$ is $D(D(V)\otimes_{\Bbbk G_{e_i}}\Bbbk \til R_{e_i})\cong \Hom_{G_{e_i}}(\til R_{e_i},V)$.
\end{enumerate}
\end{Thm}

Since the sandwich matrices of block groups can be taken to be identity matrices, Theorem~\ref{t:projindec} and Theorem~\ref{t:inj.indec} apply to describe both the projective and injective indecomposable modules of a block group.  Theorem~\ref{t:projindec} also applies if $M$ is a monoid such that each regular $\mathscr J$-class is an $\mathscr L$-class, or more generally if $M$ contains no two-element right zero semigroup (i.e., the idempotents of $M$ generate an $\mathscr R$-trivial submonoid).  Indeed, if the idempotents of $M$ generate an $\mathscr R$-trivial monoid, then it is well known to specialists that the sandwich matrices can be taken to be block diagonal where the diagonal blocks are rows of ones and such a matrix is evidently right invertible.  Let us provide some details, but in the language of modules rather than that of sandwich matrices.

Recall that $M$ acts on the right of $R_e$ by partial transformations, for each $e\in E(M)$, by restricting the right translation action. The following result can be found in~\cite[Theorem~4.8.3]{qtheor}.

\begin{Prop}\label{p:ER}
Let $M$ be a finite monoid.  Then the idempotents of $M$ generate an $\mathscr R$-trivial monoid if and only if the action of $M$ on $R_e$ is via partial injective mappings for all $e\in E(M)$.
\end{Prop}

Using Proposition~\ref{p:ER}, we show that Theorem~\ref{t:projindec} applies in this case.

\begin{Prop}\label{p:ER.rightinv}
Let $M$ be a finite monoid whose idempotents generate an $\mathscr R$-trivial monoid and $\Bbbk$ a field whose characteristic does not divide the order of the maximal subgroup $G_e$.  Then $\Coind_{G_e}(V)$ is simple for any simple $\Bbbk G_e$-module $V$.
\end{Prop}
\begin{proof}
Since $\Bbbk G_e$ is semisimple, $V$ is a direct summand in $\Bbbk G_e$.  Thus $\Coind_{G_e}(V)$ is a direct summand in $\Coind{G_e}(\Bbbk G_e)$.  Since $\Coind_{G_e}(V)$ is indecomposable, it is simple if and only if it is semisimple.  So it is enough to show that $\Coind_{G_e}(\Bbbk G_e)$ is semisimple.  But this latter module is isomorphic to $D(\Bbbk R_e)\cong \Bbbk^{R_e}$ where the module structure is given by \[mf(r)= \begin{cases} f(rm), & \text{if}\ rm\in R_e\\ 0, & \text{else}\end{cases}\]  for $f\colon R_e\to \Bbbk$, $m\in M$ and $r\in R_e$; see~\cite[Exercise~5.9]{repbook}.  Also, it is known that the socle of $D(\Bbbk R_e)\cong \Coind_{G_e}(\Bbbk G_e)$ is $\Bbbk MeD(\Bbbk R_e)$ (cf.~\cite[Propositions~4.8 and~4.19]{repbook}) so it suffices to show that $D(\Bbbk R_e)=\Bbbk MeD(\Bbbk R_e)$.

Let $\delta_r$ be the indicator mapping of $\{r\}$ for $r\in R$ (and more generally, let $\delta_A$ denote the indicator function of any $A\subseteq R_e$).  We need to show that each $\delta_r\in \Bbbk MeD(\Bbbk R_e)$.  First note that since $M$ acts on $R_e$ by partial injections and $ee=e$, it follows that $re=e$ if and only if $r=e$.  Thus $e\delta_e=\delta_{\{e\}e^{-1}} = \delta_e$ and so $\delta_e\in \Bbbk MeD(\Bbbk R_e)$.  Now if $r\in R_e$, then there exists $y\in M$ with $ry=e$.  Moreover, since $M$ acts on $R_e$ by partial injective mappings, $\{e\}y^{-1}=\{r\}$.  Thus $y\delta_e = \delta_{\{e\}y^{-1}}=\delta_r$.  This completes the proof that $D(\Bbbk R_e)=\Bbbk MeD(\Bbbk R_e)$ and hence the proof of the proposition.
\end{proof}

It follows from Propostition~\ref{p:ER.rightinv} that Theorem~\ref{t:projindec} holds for monoids whose idempotents generate an $\mathscr R$-trivial monoid.

\section{Quivers of some right Fountain monoids}

The following proposition will be used to describe the radical of a projective indecomposable module for  finite monoids satisfying the equivalent conditions of Theorem~\ref{t:projindec}.

\begin{Prop}\label{p:radical.proj.indec}
Let $M$ be a finite right Fountain monoid and $\Bbbk$ a field whose characteristic does not divide the order of the maximal subgroup $G_e$.  If $W$ is a simple $\Bbbk G_e$-module, then $\rad(\Bbbk \til L_e\otimes_{\Bbbk G_e} W)\cong\rad(\Bbbk \til L_e)\otimes_{\Bbbk G_e} W$.
\end{Prop}
\begin{proof}
Since automorphisms of a module preserve its radical, clearly $\rad(\Bbbk \til L_e)$ is a sub-bimodule.  As $\Bbbk G_e$ is semisimple, $W\cong \Bbbk G_e\eta$ for some primitive idempotent $\eta$. Then $\Bbbk \til L_e\otimes_{\Bbbk G_e} W\cong \Bbbk \til L_e\otimes_{\Bbbk G_e}\Bbbk G_e\eta\cong \Bbbk \til L_e\eta$ and so $\rad(\Bbbk\til L_e\otimes_{\Bbbk G_e} W)\cong \rad(\Bbbk \til L_e\eta)=\rad(\Bbbk \til L_e)\eta\cong \rad(\Bbbk \til L_e)\otimes_{\Bbbk G_e}\Bbbk G_e\eta\cong \rad(\Bbbk \til L_e)\otimes_{\Bbbk G_e} W$, where the equality comes from Lemma~\ref{l:rad.hit}.
\end{proof}

For an idempotent $e\in M$, let $I(f)=\{m\in M\mid f\notin MmM\}$; it is an ideal if non-empty.   If $U$ is a $\Bbbk M$-module, then put $I(f)U$ to be the span of all vectors $mu$ with $m\in I(f)$ and $u\in U$, that is, $I(f)U=\Bbbk I(f)\cdot U$.

\begin{Prop}\label{p:kill.I(f)}
Let $M$ be a right Fountain monoid and $e,f\in E(M)$.  Then we have
\[\frac{\rad(\Bbbk \til L_e)\otimes_{\Bbbk G_e}W}{I(f)\left(\rad(\Bbbk \til L_e)\otimes_{\Bbbk G_e}W\right)}\cong \left(\frac{\rad(\Bbbk \til L_e)}{I(f)\rad(\Bbbk \til L_e)}\right)\otimes_{\Bbbk G_e} W\] for any $\Bbbk G_e$-module $W$.
\end{Prop}
\begin{proof}
First note that $\rad(\Bbbk \til L_e)$ and $I(f)\rad(\Bbbk \til L_e)$ are sub-bimodules of $\Bbbk \til L_e$, the former because the radical is preserved by any module automorphism and the latter by associativity.  The isomorphism then follows from the associativity of tensor product up to isomorphism.  In detail, we have
\begin{align*}
\frac{\rad(\Bbbk \til L_e)\otimes_{\Bbbk G_e}W}{I(f)\left(\rad(\Bbbk \til L_e)\otimes_{\Bbbk G_e}W\right)} &\cong \frac{\Bbbk M}{\Bbbk I(f)}\otimes_{\Bbbk M}\left(\rad(\Bbbk \til L_e)\otimes_{\Bbbk G_e} W\right)\\
&\cong \left(\frac{\Bbbk M}{\Bbbk I(f)}\otimes_{\Bbbk M} \rad(\Bbbk \til L_e)\right)\otimes_{\Bbbk G_e} W\\
&\cong  \left(\frac{\rad(\Bbbk \til L_e)}{I(f)\rad(\Bbbk \til L_e)}\right)\otimes_{\Bbbk G_e} W
\end{align*}
as required.
\end{proof}

We now provide a reduction for computing the quiver of $\Bbbk M$ when $M$ is a right Fountain monoid and $\Bbbk$ is a splitting field for each maximal subgroup of $M$ whose characteristic divides the order of no maximal subgroup of $M$ provided that the equivalent conditions of Theorem~\ref{t:projindec} hold.  If $e,f\in M$, then $f\rad(\Bbbk \til L_e)/fI(f)\rad(\Bbbk \til L_e)$ is a $\Bbbk G_f$-$\Bbbk G_e$-bimodule and hence a left $\Bbbk [G_f\times G_e]$-module in the natural way.  

\begin{Thm}\label{t:reduction}
Let $M$ be a finite right Fountain monoid and $\Bbbk$ a field. Suppose that:
\begin{enumerate}
\item  the characteristic of $\Bbbk$ divides the order of no maximal subgroup of $M$;
\item $\Bbbk$ is a splitting field for each maximal subgroup of $M$; and
\item  the equivalent conditions of Theorem~\ref{t:projindec} hold.
\end{enumerate}
Let $e,f\in E(M)$ and let $W$ be a simple $\Bbbk G_f$-module and $V$ a simple $\Bbbk G_e$-module.
Then the number of arrows in the quiver of $\Bbbk M$ from the isomorphism class of $\Coind_{G_e}(V)$ to the isomorphism class of $\Coind_{G_f}(W)$  is the multiplicity of $W\otimes D(V)$ as an irreducible constituent in the $\Bbbk [G_f\times G_e]$-module  $f\rad(\Bbbk \til L_e)/fI(f)\rad(\Bbbk \til L_e)$.
\end{Thm}
\begin{proof}
By Theorem~\ref{t:projindec}, we have that $\Bbbk \til L_e\otimes_{\Bbbk G_e} V\to \Coind_{G_e}(V)$ is a projective cover.  Therefore, we have that
\begin{equation}\label{eq:iso.red}
\begin{split}
\Ext^1_{\Bbbk M}(\Coind_{G_e}(V),\Coind_{G_f}(W))&\cong \Hom_{\Bbbk M}(\rad(\Bbbk \til L_e\otimes_{\Bbbk G_e}V),\Coind_{G_f}(W))\\ &\cong \Hom_{\Bbbk M}(\rad(\Bbbk \til L_e)\otimes_{\Bbbk G_e} V,\Coind_{G_f}(W))
\end{split}
\end{equation}
 where first isomorphism uses the long exact sequence for $\Ext$ and the second uses Proposition~\ref{p:radical.proj.indec}.  Since $\Coind_{G_f}(W)$ is a $\Bbbk M/\Bbbk I(f)$-module, we then have, in light of Proposition~\ref{p:kill.I(f)}, that the right hand side of \eqref{eq:iso.red} is isomorphic to
 \begin{equation}\label{eq:iso.red.two}
 \begin{split}
 \Hom_{\Bbbk M/\Bbbk I(f)}\left(\frac{\rad(\Bbbk \til L_e)\otimes_{\Bbbk G_e} V}{I(f)\left(\rad(\Bbbk\til L_e)\otimes_{\Bbbk G_e} V\right)},\Coind_{G_f}(W)\right)\cong\\
 \qquad \Hom_{\Bbbk M/\Bbbk I(f)}\left(\left(\frac{\rad(\Bbbk \til L_e)}{I(f)\rad(\Bbbk \til L_e)}\right)\otimes_{\Bbbk G_e} V,\Coind_{G_f}(W)\right).
 \end{split}
 \end{equation}
Let $\eta_W$ be a primitive idempotent with $\Bbbk G_f\eta_W\cong W$ and $\eta_V$ a primitive idempotent with $\Bbbk G_e\eta_V\cong V$.  Applying that coinduction is right adjoint to restriction yields that the right hand side of \eqref{eq:iso.red.two} is isomorphic to
\begin{equation}
\Hom_{\Bbbk G_f}\left(\left(\frac{f\rad(\Bbbk \til L_e)}{fI(f)\rad(\Bbbk \til L_e)}\right)\otimes_{\Bbbk G_e} V,W\right)\cong D\left(\eta_W\left[\frac{f\rad(\Bbbk \til L_e)}{fI(f)\rad(\Bbbk \til L_e)}\right]\eta_V\right).
\end{equation}
Since $\Bbbk$ is a splitting field for $G_f$ and $G_e$, we have that $\eta_W\otimes \eta_V$ is the primitive idempotent of $\Bbbk G_f\otimes \Bbbk G_e\cong \Bbbk [G_f\times G_e]$ corresponding to $W\otimes D(V)$.

This completes the proof that $\dim \Ext_{\Bbbk M}^1(\Coind_{G_e}(V),\Coind_{G_f}(W))$ is the multiplicity of $W\otimes D(V)$ as an irreducible constituent of the $\Bbbk [G_f\times G_e]$-module $f\rad(\Bbbk \til L_e)/fI(f)\rad(\Bbbk \til L_e)$.
\end{proof}

\subsection{Quivers of monoids whose idempotents generate an $\mathscr R$-trivial monoid}
In order to compute the quiver of $\Bbbk M$, for a finite monoid $M\in \mathbf{ER}$ over a field $\Bbbk$ whose characteristic does not divide the order of any maximal subgroup of $M$ and which is a splitting field for each maximal subgroup, we need to compute the kernel of the natural homomorphism $\p_{\Bbbk G_e}\colon \Ind_{G_e}(\Bbbk G_e)\to \Coind_{G_e}(\Bbbk G_e)$. Recall that $\Ind_{G_e}(\Bbbk G_e)=\Bbbk L_e$.    Since $\Coind_{G_e}(\Bbbk G_e)\cong D(\Bbbk R_e)$ is semisimple by Proposition~\ref{p:ER.rightinv}, $\ker \p_{\Bbbk G_e}=\rad(\Bbbk L_e)$ and $\p_{\Bbbk G_e}$ is surjective.

\begin{Prop}\label{p:ER.rad.L}
Let $M$ be a finite monoid whose idempotents generate an $\mathscr R$-trivial monoid, $e\in E(M)$ and $\Bbbk$ a field whose characteristic does not divide the order of $G_e$.  Then $\rad(\Bbbk L_e)$ is spanned by all differences $x-y$ with $x,y\in L_e$ such that $x$ and $y$ act as the same partial injective mapping on the right of $R_e$.
\end{Prop}
\begin{proof}
Recall that if $x\in L_e$ and $r\in R_e$, then
\[(\p_{\Bbbk G_e}(x))(r)= \begin{cases}rx, & \text{if}\ rx\in R_e\\ 0, & \text{else}\end{cases}\] and so if $x,y\in L_e$ act as the same partial injection on the right of $R_e$, then $x-y\in \ker \p_{\Bbbk G_e}=\rad(\Bbbk L_e)$.  (Note that if $r\in R_e$, $x\in L_e$, then $rx\in R_e$ if and only if $rx\in G_e$ by stability.)

Since $\p_{\Bbbk G_e}$ is surjective by semisimplicty of $\Coind_{G_e}(\Bbbk G_e)\cong D(\Bbbk R_e)$ (see Proposition~\ref{p:ER.rightinv}), it follows that $\ker \p_{\Bbbk G_e}$ has dimension $|L_e|-|R_e|$. Let $T$ be a complete set of equivalence class representatives for the equivalence relation $\sim$ on $L_e$ given by $x\sim y$ if they act the same on the right of $R_e$ by partial injections.  If $x\in L_e$, let $\ov x\in T$ be the representative of the $\sim$-class of $x$.  Then the span of the differences $x-y$ with $x\sim y$ has basis the non-zero elements of the form $x-\ov x$.  There are exactly $|L_e|-|T|$ such elements.  So to complete the proof of the proposition, it suffices to show that $|T|=|R_e|$.

We claim that if $x\in T$, there is a unique element $r_x\in R_e$ such that $r_xx=e$.  Indeed, since $x\in L_e$, we have $mx=e$ for some $m\in M$.  Then $emx=ee=e$ and so $r_x=em\in R_e $.  Uniqueness follows because $x$ acts as a partial injection on the right of $R_e$.

We thus have a mapping $\psi\colon T\to R_e$ given by $\psi(x)= r_x$.  We claim that $\psi$ is a bijection.  If $r\in R_e$, then $ra=e$ for some $a\in M$ and so $rae=ee=e$, whence $ae\in L_e$.  If $x=\ov{ae}$, then $rx=rae=e$ and so $r=\psi(x)$.  Thus $\psi$ is onto.

Now suppose that $\psi(x)=r=\psi(y)$.  Then $rx=e=ry$. As $r\in R_e$ and $x\in L_e$, we deduce that $rxr=er=r$ and $xrx=xe=x$.  Similarly, $ryr=r$ and $yry=y$.   We claim that if $r'\in R_e$, then $r'x\in R_e$ if and only if $r'y\in R_e$, and if $r'x,r'y\in R_e$, then $r'x=r'y$.  Indeed, if $r'x\in R_e$, then $r'xrx=r'x$ and so $r'xr\in R_e$. But $r'(xr)=r'xr(xr)$ and hence, since $xr$ acts on the right of $R_e$ as a partial injection, we have $r'=r'xr$.  Thus $r'yr=r'xryr=r'xr=r'$.  Therefore, $r'y\in R_e$.  Also, since $r'xr=r'yr$ and $r$ acts on the right of $R_e$ as a partial injection, we must have $r'x=r'y$.  Similarly, if $r'y\in R_e$, then $r'x\in R_e$ and $r'x=r'y$.  Thus $x\sim y$ and hence $x=y$, as $x,y\in T$.  Therefore,  $\psi$ is a bijection.  This completes the proof.
\end{proof}

Our goal now is to give an explicit description of the $\Bbbk [G_f\times G_e]$-module $f\rad(\Bbbk \til L_e)/fI(f)\rad(\Bbbk \til L_e)$ from Theorem~\ref{t:reduction} when $M\in \mathbf{ER}$.  So for the remainder of this subsection, $M$ will denote a finite monoid whose idempotents generate an $\mathscr R$-trivial monoid, $\Bbbk$ will be a field whose characteristic divides the order of no maximal subgroup of $M$ and which is a splitting field for all maximal subgroups and $e,f\in E(M)$.  Note that $Me\setminus L_e=I(e)e$ and that $\til L_e\setminus L_e=I(e)e\cap \til L_e$ by stability.

Let $\sim$ be the congruence on $M$ defined by $m\sim n$ if they act on the right of $R_e$ as the same partial injection.
Let $\equiv$ be the least equivalence relation on $fMe$ such that:
\begin{enumerate}
  \item $x\equiv y$ if $x,y\in fI(f)I(e)e$
  \item $x\equiv y$ if $x=zu$ and $y=zv$ with $z\in fI(f)$ and $u,v\in L_e$ with $u\sim v$.
\end{enumerate}

\begin{Prop}\label{p:control.eq}
The equivalence relation $\equiv$ on  $fMe$ enjoys the following properties.
\begin{enumerate}
  \item ${\equiv}\subseteq{\sim}$.
  \item If $x\equiv y$, then $x\in L_e$ if and only if $y\in L_e$.
\end{enumerate}
\end{Prop}
\begin{proof}
Elements of $I(e)$ all act on $R_e$ as the empty function, so the first property defining $\equiv$ is satisfied by $\sim$.  The second property is clear since $\sim$ is a congruence.  Thus ${\equiv}\subseteq {\sim}$.  If $x\in L_e$ with $mx=e$, then $em\in R_e$ and $emx=e$.  Thus $x$ does not act as the empty mapping on $R_e$ and so $\sim$ separates $L_e$ from $I(e)$.  As ${\equiv}\subseteq {\sim}$, we deduce that $\equiv$ separates $fM\cap L_e$ from $fMe\setminus L_e$.
\end{proof}

We next want to show that $\equiv$ is $G_f\times G_e$-stable.

\begin{Prop}\label{p:equivariant}
The equivalence relation $\equiv$ on $fMe$ is $G_f\times G_e$-stable. That is, if $g\in G_f$, $h\in G_e$ and $x\equiv y$.  Then $gxh\inv\equiv gyh\inv$.
\end{Prop}
\begin{proof}
Define an equivalence relation $\smallfrown$ on $fMe$ by $x\smallfrown y$ if and only if $gxh\inv\equiv gyh\inv$ for all $g\in G_f$ and $h\in G_e$.  Clearly, ${\smallfrown}\subseteq {\equiv}$ and $\smallfrown$ is $G_f\times G_e$-stable.  We claim that $\smallfrown$ satisfies the defining properties of $\equiv$ and hence the reverse inclusion holds.  Since $fI(f)I(e)e$ is $G_f\times G_e$-invariant, the first property is clear.  Since $L_e$ is invariant under right multiplication by elements of $G_e$, $\sim$ is a congruence on $M$ and $fI(f)$ is invariant under left multiplication by elements of $G_f$, the second property is also clear.
\end{proof}

It follows that $fMe/{\equiv}$ is a $G_f\times G_e$-set via the action $(g,h)[x]_{\equiv}=[gxh^{-1}]_{\equiv}$.  Let $C$ be the $\equiv$-class of $fI(f)I(e)e$ (which is possibly empty if $e$ or $f$ belongs to the minimal ideal).  Then $X=(fMe/{\equiv})\setminus \{C\}$ is a $G_f\times G_e$-invariant subset of $fMe/{\equiv}$. Let $U=\Bbbk X$ be the corresponding permutation module.  The next proposition shows that $fI(f)e\setminus C\subseteq \til L_e$.

\begin{Prop}\label{p:kick.out}
Suppose that $z\in fI(f)e$.  If $z\notin \til L_e$, then $z\in fI(f)I(e)e$.  Consequently, $fI(f)e\setminus C\subseteq \til L_e$.
\end{Prop}
\begin{proof}
If $z\notin \til L_e$, then $z\in \til L_u$ with $u\in E(M)$ and $u\notin L_e$.  Then from $ze=z$, we obtain $ue=u$.  So $u\in I(e)e$.  Thus $z=zu\in fI(f)I(e)e$.
\end{proof}

As ${\equiv}\subseteq {\sim}$ by Proposition~\ref{p:control.eq}, we have that $\sim$ descends to an equivalence relation on $fMe/{\equiv}$.  Also, as $\sim$ is a congruence on $M$, we have that $\sim$ is a $G_f\times G_e$-stable equivalence relation on $fMe/{\equiv}$.    Fix a transversal $T$ to the restriction of $\sim$ to $fM\cap L_e$ and write $\ov x$ for the element of $T$ in the $\sim$-class of $x\in fM\cap L_e$. Let $U^\flat$ be the subspace of $U$ spanned by all $[z]_{\equiv}$ with $z\in (fM\cap \til L_e)\setminus (L_e\cup C)$ and all $[x]_{\equiv}-[y]_{\equiv}$ such that $x,y\in fM\cap L_e$ with  $x\sim y$.  Since $\sim$ is $G_f\times G_e$-stable and $(fM\cap \til L_e)\setminus (L_e\cup C)$ is $G_f\times G_e$-invariant, it is immediate that $U^\flat$ is a $\Bbbk[G_f\times G_e]$-submodule.  The proof of the following proposition is routine.

\begin{Prop}\label{p:basis.flat.U}
The vector space $U^\flat$ has basis the  elements of the form $[x]_{\equiv}-[\ov x]_{\equiv}$ with $x\in (fM\cap L_e)\setminus T$ and $[z]_{\equiv}$ with $z\in (fM\cap \til L_e)\setminus (L_e\cup C)$.
\end{Prop}

Our goal is to show that $U^\flat\cong f\rad(\Bbbk \til L_e)/fI(f)\rad(\Bbbk \til L_e)$ as a $\Bbbk[G_f\times G_e]$-module.

\begin{Prop}\label{p:iso.bimod}
There is an isomorphism \[U^\flat\cong f\rad(\Bbbk \til L_e)/fI(f)\rad(\Bbbk \til L_e)\] of $\Bbbk[G_f\times G_e]$-modules.
\end{Prop}
\begin{proof}
 It follows from Proposition~\ref{p:ER.rad.L} that $f\rad(\Bbbk \til L_e)$ has basis consisting of the elements $x-\ov x$ with $x\in (fM\cap L_e)\setminus T$ and $y\in fM\cap \til L_e\setminus L_e$.  Define $\psi\colon f\rad(\Bbbk \til L_e)\to U^\flat$ on these basis elements by $\psi(x-\ov x) = [x]_{\equiv}-[\ov x]_{\equiv}$ if $x\in (fM\cap L_e)\setminus T$ , $\psi(y)=[y]_{\equiv}$ if $y\in fM\cap \til L_e\setminus (L_e\cup C)$ and $\psi(y)=0$ if $y\in C$.  Note that if $x\in T$, then $\psi(x-\ov x)=0=[x]_{\equiv}-[\ov x]_{\equiv}$.   We claim that $\psi$ is a $\Bbbk [G_f\times G_e]$-module homomorphism.  Let $g\in G_f$ and $h\in G_e$.  If $x\in fM\cap L_e$, then $gxh\inv, g\ov xh\inv\in fM\cap L_e$ and $gxh\inv \sim g\ov xh\inv$.  Therefore,  $\psi(g(x-\ov x)h\inv) = \psi(gxh\inv -\ov{gxh\inv})-\psi(g\ov xh\inv -\ov{gxh\inv})=[gxh\inv]_{\equiv}-[\ov{gxh\inv}]_{\equiv}-([g\ov xh\inv]_{\equiv}-[\ov{gxh\inv}]_{\equiv})=g([x]_{\equiv}-[\ov x]_{\equiv})h\inv =g\psi(x-\ov x)h\inv$.  If $y\in fM\cap \til L_e\setminus L_e$, then $gyh\inv \in fM\cap \til L_e\setminus L_e$ and $y\in C$ if and only if $gyh\inv \in C$.  Thus $\psi(gyh\inv)=g\psi(y)h\inv$, trivially, in this case.

We claim that $fI(f)\rad(\Bbbk \til L_e)\subseteq \ker \psi$.  Note that $fI(f)\rad(\Bbbk \til L_e)$ is spanned by elements of the form $z(x-\ov x)$ with $z\in fI(f)$ and $x\in (fM\cap L_e)\setminus T$ and $zy\in \til L_e$ with $z\in fI(f)$ and $y\in fM\cap \til L_e\setminus L_e$.  In the latter case, $zy\in fI(f)I(e)e$ and so $\psi(zy)=0$ if $zy\in \til L_e$ (of course if $zy\notin \til L_e$,  then there is nothing to prove).  In the former case, $zx\equiv z\ov x$ by definition of the equivalence relation.  There are  a couple of cases. If $zx\notin \til L_e$, then $zx\in fI(f)I(e)e$ by Proposition~\ref{p:kick.out} and hence $z\ov x\in C$.  So $\psi(z(x-\ov x))=0$.  The same occurs if $z\ov x\notin \til L_e$.  Thus we may assume that $zx,z\ov x\in \til L_e$.  Since $zx\equiv z\ov x$, it follows from Proposition~\ref{p:control.eq} that either $zx,z\ov x\in L_e$ or $zx,z\ov x\in \til L_e\setminus L_e$.  In the former case, we then have $\ov{zx}=\ov {z\ov x}$ because $\sim$ is a congruence and hence   $\psi(z(x-\ov x)) = \psi(zx-\ov{zx})-\psi(z\ov x-\ov{zx})=[zx]_{\equiv}-[\ov{zx}]_{\equiv} - ([z\ov x]_{\equiv}-[\ov {zx}]_{\equiv})=0$.  In the case that $zx,z\ov x\in \til L_e\setminus L_e$, we have that $\psi(z(x-\ov x)) = \psi(zx)-\psi(z\ov x) = [zx]_{\equiv}-[z\ov x]_{\equiv}=0$.  Thus in all cases, $z(x-\ov x)\in \ker \psi$ and so $\psi$ descends to a well-defined $\Bbbk[G_f\times G_e]$-module homomorphism $\Psi\colon f\rad(\Bbbk \til L_e)/fI(f)\rad(\Bbbk \til L_e)\to U^\flat$.   Also, we shall use without comment that $\Psi$ maps the coset of $x$ to $0$ if $x\in C\cap \til L_e$.

Put $K=fI(f)\rad(\Bbbk \til L_e)$ and define $\rho\colon U^\flat\to f\rad(\Bbbk \til L_e)/fI(f)\rad(\Bbbk \til L_e)$ on the basis from Proposition~\ref{p:basis.flat.U} by $\rho([x]_{\equiv}-[\ov x]_{\equiv})=x-\ov x+K$ if $x\in (fM\cap L_e)\setminus T$ and $\rho([y]_{\equiv}) = y+K$ if $y\in (fM\cap \til L_e)\setminus (L_e\cup C)$.  We must show that $\rho$ is well defined.     We can view $f\Bbbk \til L_e$ as the quotient of $\Bbbk [fMe]$ by the subspace spanned by $fMe\setminus (fM\cap \til L_e)$ and $K$ as a subspace of $f\Bbbk \til L_e$.  Let $K'$ be the preimage of $K$ under the quotient $\Bbbk[fMe]\to f\Bbbk \til L_e$.    Notice that $fI(f)I(e)e\subseteq K'$ and so $x+K'=y+K'$ for all $x,y\in fI(f)I(e)e$.  Suppose that $x,y\in L_e$ with $x\sim y$ and $z\in fI(f)$.  Then $x-y\in \rad(\Bbbk \til L_e)$ by Proposition~\ref{p:ER.rad.L} and so $z(x-y)\in K'$.  Thus $zx+K'=zy+K'$.   We deduce that $a\equiv b$ implies $a+K'=b+K'$.
It follows immediately that $\rho$ is well defined.  

We claim that $\Psi$ and $\rho$ are inverse mappings.  Suppose that $x\in (fM\cap L_e)\setminus T$.  Then $\rho(\Psi(x-\ov x+K))=\rho(\psi(x-\ov x)) = \rho([x]_{\equiv}-[\ov x]_{\equiv}) = x-\ov x+K$ and if $y\in (fM\cap \til L_e)\setminus (L_e\cup C)$, then $\rho(\Psi(y+K)) =\rho\psi(y)=\rho([y]_{\equiv}) = y+K$.  Thus $\rho\circ \Psi$ is the identity.

Similarly, if $x\in (fM\cap L)\setminus T$, then $\Psi(\rho([x]_{\equiv}-[\ov x]_{\equiv})) = \Psi(x-\ov x+K)=[x]_{\equiv}-[\ov x]_{\equiv}$ and if $y\in (fM\cap \til L_e)\setminus (L_e\cup C)$, then $\Psi(\rho([y]_{\equiv}))= \Psi(y+K)=[y]_{\equiv}$.  This completes the proof that $\Psi$ and $\rho$ are inverse isomorphisms.
\end{proof}

The main result of this section is the next theorem, which is immediate from Theorem~\ref{t:reduction} and the previous results.

\begin{Thm}\label{t:ER.quiver}
Let $M$ be a finite monoid whose idempotents generate an $\mathscr R$-trivial monoid and $\Bbbk$ a field whose characteristic does not divide the order of any maximal subgroup of $M$ and over which each maximal subgroup of $M$ splits.  Let $e_1,\ldots, e_n$ be a complete set of idempotent representatives of the regular $\mathscr J$-classes of $M$.  Then the quiver of $\Bbbk M$ is isomorphic to the directed graph that has vertices indexed by pairs $(e_i,[V])$ where $V$ is a simple $\Bbbk G_{e_i}$-module and edges as follows.  If $(e,[V])$ and $(f,[W])$ are two vertices, then the number of arrow from $(e,[V])$ to $(f,[W])$ is the multiplicity of $W\otimes D(V)$ as an irreducible constituent of the $\Bbbk [G_f\times G_e]$-module $U^\flat$ defined as follows.

Let $\equiv$ be the least equivalence relation on $fMe$ such that:
\begin{enumerate}
\item $x\equiv y$ if $x,y\in fI(f)I(e)e$;
\item $zx\equiv zy$ if $x,y\in L_e$, $z\in fI(f)$ and $x,y$ act as the same partial injection on the right of $R_e$.
\end{enumerate}
Let $X$ be the set of equivalence classes of elements of $fMe$ not meeting $fI(f)I(e)e$; it is naturally a $G_f\times G_e$-set.  Let $U^\flat$ be the submodule of the permutation module $\Bbbk X$ spanned by differences $[x]_{\equiv}-[y]_{\equiv}$ with $x,y\in L_e$ such that $x$ and $y$ act as the same partial injection on the right of $R_e$ and by those $[z]_{\equiv}\in X$ such that $z\in (fM\cap \til L_e)\setminus L_e$.
\end{Thm}

Recall that a finite monoid $M$ is \emph{aperiodic} if all its maximal subgroups are trivial.

\begin{Cor}\label{c:aperiodic.ER}
Let $M$ be a finite aperiodic monoid whose idempotents generate an $\mathscr R$-trivial monoid and let $\Bbbk$ be a field.  Then the quiver of $\Bbbk M$ is isomorphic to the directed graph that has one vertex for each regular $\mathscr J$-class of $M$ and with edge set as follows.  If $e,f\in E(M)$, then the number of arrows from $J_e$ to $J_f$ is $|X|-|R_e\cap Mf|$ where $X$ is defined in the following way.  Let $\equiv$ be the least equivalence relation on $fMe$ such that:
\begin{enumerate}
\item $x\equiv y$ if $x,y\in fI(f)I(e)e$;
\item $zx\equiv zy$ if $x,y\in L_e$, $z\in fI(f)$ and $x,y$ act as the same partial injection on the right of $R_e$.
\end{enumerate}
Then $X$ is the set of equivalence classes of elements of $fM\cap \til L_e$ not meeting $fI(f)I(e)e$.
\end{Cor}
\begin{proof}
This follows from Theorem~\ref{t:ER.quiver}, Proposition~\ref{p:ER.rad.L} and the observation that  $\Bbbk L_e/\rad(\Bbbk L_e)\cong D(\Bbbk R_e)$, as $D(\Bbbk R_e)$ is the simple module corresponding to the $\mathscr J$-class $J_e$, and so if $T$ is a transversal to $\sim$ on $fM\cap L_e$, then $|T|=\dim fD(\Bbbk R_e) = |R_e\cap Mf|$.
\end{proof}

\subsection{Quivers of block groups}

Recall that a monoid $M$ is a \emph{block group} if, for each $a\in M$, there is at most one $b\in M$ such that $aba=a$ and $bab=b$.  The power set of a finite group is an important example of a block group.  A finite monoid is a block group if and only if its idempotents generate a $\mathscr J$-trivial monoid (and hence a monoid that is both $\mathscr R$-trivial and $\mathscr L$-trivial); in particular, block groups are Fountain.  Thus Theorem~\ref{t:ER.quiver} can be applied to compute the quiver of the algebra of a block group.  However, there are a number of simplifications in this case.

Let $M$ be a block group and fix idempotents $e,f\in M$. Let $\Irr(e,f)$ be the set of elements $x\in (\til R_f\cap \til L_e)\setminus (R_f\cup L_e)$ such that if $x=yz$ with $y\in \til R_f$ and $z\in \til L_e$, then $y\in R_f$ or $z\in L_e$.

\begin{Prop}\label{p:jbelow}
Let $M$ be a block group and $e,f\in E(M)$.  If $x\in fMe$, then the following are equivalent.
\begin{enumerate}
\item $x\in \Irr(e,f)$.
\item $x$ is not regular and $x\notin fI(f)I(e)e$
\end{enumerate}
\end{Prop}
\begin{proof}
Suppose that $x\in \Irr(e,f)$.  Then since $x\in \til R_f\cap \til L_e$ and $x\notin R_f\cup L_e$, we deduce that $x$ is not regular.   Suppose that $x=yz$ with $y\in fI(f)$ and $z\in I(e)e$.  Then if $a\in E(M)$ with $a\in \til R_y$, we have that $ay=y$ and so $ax=ayz=yz=x$.  Therefore, $af=f$ because $x\in \til R_f$.  But $fy=y$ implies $fa=a$ as $a\in \til R_y$.  Thus $f\mathrel{\mathscr R} a$ and so $y\in \til R_f$.  Similarly, $z\in \til L_e$.  Then by definition of irreducibility, we obtain that $y\in R_f$ or $z\in L_e$, contradicting that $y\in fI(f)$ and $z\in I(e)e$.  This shows that the first item implies the second.

Next suppose that $x\notin fI(f)I(e)e$ and $x$ is not regular. Then $x\notin J_f\cup J_e$ because $x$ is not regular.  Let $a\in E(M)$ with $a\in \til R_x$.  Then from $fx=x$, we must have $fa=a$.  If $a\in fI(f)$, then from $a(xe)=x$ and the non-regularity of $x$, we have that $x\in fI(f)I(e)e$, a contradiction.  Thus $J_a=J_f$ and so from $fa=a$ we have that $R_f=R_a$.  Thus $x\in \til R_f$.  Similarly, $x\in \til L_e$.  Suppose that $x=yz$ with $y\in \til R_f$ and $z\in \til L_e$.  Then $fy=y$, $ze=z$ and so from $x\notin fI(f)I(e)e$, we deduce that $y\in R_f$ or $z\in L_e$.  This completes the proof that $x$ is irreducible.
\end{proof}

Note that it follows from the second item of Proposition~\ref{p:jbelow} that $\Irr(e,f)$ is a $G_f\times G_e$-set under the action $(g,h)x= gxh^{-1}$.


\begin{Thm}\label{t:block.group.quiver}
Let $M$ be a block group and $e,f\in E(M)$. Let $\Bbbk$ be field such that the characteristic of $\Bbbk$ divides the order of no maximal subgroup of $M$ and $\Bbbk$ is a splitting field for each maximal subgroup of $M$.  Let $e_1,\ldots, e_n$ be a complete set of idempotent representatives of the regular $\mathscr J$-classes of $M$.  Then the quiver of $\Bbbk M$ is isomorphic to the directed graph that has vertices indexed by pairs $(e_i,[V])$ where $V$ is a simple $\Bbbk G_{e_i}$-module and edges as follows.  If $(e,[V])$ and $(f,[W])$ are two vertices, then the number of arrow from $(e,[V])$ to $(f,[W])$ is the multiplicity of $W\otimes D(V)$ as an irreducible constituent of the $\Bbbk [G_f\times G_e]$-permutation module $\Bbbk\Irr(e,f)$.
\end{Thm}
\begin{proof}
We apply Theorem~\ref{t:ER.quiver}.  First note  that by the dual of Proposition~\ref{p:ER}, $M$ acts by partial injections on the left of $L_e$.  Hence, if $x,y\in L_e$ and $r\in R_e$ with $rx=ry\in R_e$, then $x=y$.  It follows that the equivalence relation $\sim$ is equality on $L_e$ and hence $\equiv$ is simply the equivalence relation on $fMe$ identifying all elements of $fI(f)I(e)e$ and $U^\flat$ is the $\Bbbk[G_f\times G_e]$-module with basis the elements of $(fM\cap \til L_e)\setminus (L_e\cup fI(f)I(e)e)$. But this is $\Irr(e,f)$ by  Proposition~\ref{p:jbelow}.
\end{proof}

The following theorem generalizes one of the main result of~\cite{Jtrivialpaper} from $\mathscr J$-trivial monoids to arbitrary aperiodic block groups.

\begin{Cor}
Let $M$ be an aperiodic block group and $\Bbbk$ a field.  Fix $e_1,\ldots,e_n$ as set of idempotent representatives of the regular $\mathscr J$-classes of $M$.  Then the quiver of $\Bbbk M$ is isomorphic to the quiver with vertex set $e_1,\ldots, e_n$ and with $|\Irr(e_i,e_j)|$ arrows from $e_i$ to $e_j$.
\end{Cor}

\end{document}